\documentclass[twoside,leqno,symbols-for-thanks, 
]{rmi}
\usepackage[colorlinks,linkcolor=black,citecolor=black,urlcolor=black, linktocpage=true]{hyperref}

\usepackage{srcltx}
\usepackage{enumerate}
\usepackage{amsthm}
\usepackage[english]{babel}
\numberwithin{equation}{section}
\setinitialpage{1}
\title[On the factorization of iterated polynomials]{On the factorization of iterated polynomials}

\author[L. Reis]{Lucas Reis}

\address[lucasreismat@gmail.com]{Instituto de Ci\^{e}ncias Matem\'{a}ticas e de Computa\c{c}\~{a}o\\
Universidade de S\~{a}o Paulo\\
USP\\
S\~{a}o Carlos, SP\\
13560-970\\
Brazil\\}

\amsclassification[37P05]{12E05}


\newcommand{\F}{\mathbb{F}}

\newcommand{\ord}{\mathrm{ord}}

\newcommand{\ou}{\mathcal O}
\newtheorem{theorem}{Theorem}[section]
\newtheorem{conjecture}[theorem]{Conjecture}
\newtheorem{proposition}[theorem]{Proposition}
\newtheorem{definition}[theorem]{Definition}
\newtheorem{problem}[theorem]{Problem}
\newtheorem{lemma}[theorem]{Lemma}
\newtheorem{corollary}[theorem]{Corollary}

\newtheorem{remark}[theorem]{Remark}

\begin{document}
%
%

\begin{abstract}
Let $\F_q$ be the finite field with $q$ elements and $f, g\in \F_q[x]$ be polynomials of degree at least one. This paper deals with the asymptotic growth of certain arithmetic functions associated to the factorization of the iterated polynomials $f(g^{(n)}(x))$ over $\F_q$, such as the largest degree of an irreducible factor and the number of irreducible factors. In particular, we provide significant improvements on the results of D.~G\'{o}mez-P\'{e}rez, A.~Ostafe and I.~Shparlinski (2014). 
\end{abstract}

\section{Introduction}
For a field $K$ and a polynomial $f\in K[x]$, let $f^{(0)}(x)=x$ and $f^{(n)}(x)=f(f^{(n-1)}(x))$ for $n\ge 1$. The polynomial $f^{(n)}(x)$ is called the $n$-th iterate of $f$. Algebraic aspects of $f^{(n)}(x)$ have been extensively considered by many authors in the past few years \cite{A05,AM00,J08,JB12,K85}; in most of the cases, the main object of study is the class of polynomials $f\in K[x]$ for which its iterates $f^{(n)}(x)$ are all irreducible over $K$. Such polynomials are called \emph{stable}. In the case $K$ a finite field, a recent work~\cite{HBM} extends the notion of stability to finite sets of polynomials $\{f_1, \ldots, f_r\}$. More specifically, the authors explore the sets of quadratic polynomials $\mathcal C=\{f_1, \ldots, f_r\}$ over $\F_q$ such that any polynomial obtained by compositions of elements in $\mathcal C$ is irreducible. In~\cite{GOS}, the authors explore further arithmetic properties of the factorization of $f^{(n)}(x)$ when $K=\F_q$ is the finite field with $q$ elements. They obtain lower bounds on the values of some arithmetic functions related to the factorization of $f^{(n)}(x)$ such as the number of irreducible factors and the degree of its squarefree part. For functions $A, B:\mathbb N\to \mathbb R_{>0}$, we write $A\ll B$ (or equivalently $B\gg A$) if $A(n)\le c\cdot B(n)$ for some $c>0$ and $A=o(B)$ if $\lim\limits_{n\to \infty}\frac{A(n)}{B(n)}=0$. In addition, we write $A\approx B$ if $A\ll B$ and $A\gg B$. The main results in~\cite{GOS} are Theorems 9, 10 and 11 and they can be compiled as follows.

\begin{theorem}\label{thm:shpa}
For any fixed $\varepsilon>0$ and a positive integer $d$ such that $\gcd(d, q)=1$, all but $o(q^{d+1})$ polynomials $f\in \F_q[x]$ of degree $d$ satisfy the following:

\begin{enumerate}[(i)]
\item If $\Delta_n(f)$ is the degree of the squarefree part of $f^{(n)}(x)$, then $\Delta_n(f)\gg n^{1-\varepsilon}$. Moreover, the implicit constant in the previous inequality can be taken uniformly on $q$.

\item If $r_n(f)$ is the number of distinct irreducible factors of $f^{(n)}(x)$, we have that $r_n(f)\ge (0.5+o(1))\cdot n$ when $n\to \infty$ provided that
$$n\le \left\lceil\left(\frac{1}{2\log d}-\varepsilon\right)\log q\right\rceil.$$
\end{enumerate}
Additionally, if $f\in \F_q[x]$ has degree $d$ with $\gcd(d, q)=1$ and $f$ is not of the form $ax^d$, then the largest degree $D_n(f)$ of an irreducible factor of $f^{(n)}(x)$ satisfies
\begin{equation*}\label{eq:shp}
D_n(f)>\frac{(n-1)\log d-\log 2}{\log q}\approx n.
\end{equation*}
\end{theorem}

In particular, the functions $\Delta_n(f), r_n(f)$ and $D_n(f)$ grow (roughly) at least linearly with respect to $n$ under not too restricted conditions. It is worth mentioning that the exclusion of $o(q^{d+1})$ polynomials in Theorem~\ref{thm:shpa} is due to the character sums techniques employed in its proof.

In this paper, we explore similar questions to the ones discussed in~\cite{GOS} over a more general class of iterated polynomials. For polynomials $f, g\in \F_q[x]$, we define the following sequence of polynomials over $\F_q$:
$$P_0[f, g](x)=f(x), \quad P_n[f, g](x)= f(g^{(n)}(x)),\quad n\ge 1.$$
We call $P_n[f, g](x)$ the $n$-th $g$-iterate of $f$. Motivated by Question 18.9
in~\cite{arxiv}, we introduce the following arithmetic functions associated to the factorization of $P_n[f, g](x)$ over $\F_q$.
\begin{definition}\label{def:main}
For $f, g\in\F_q[x]$, let
$$P_n[f, g](x)=p_{1, n}(x)^{e_{1, n}}\cdots p_{N_n, n}(x)^{e_{N_n, n}},\quad n\ge 0,$$
be the factorization of $P_n[f, g](x)$ into irreducible polynomials over $\F_q$. For each $n\ge 0$, we consider the following arithmetic functions:

\begin{enumerate}[(a)]
\item $E_{f, g}(n):=\max\limits_{1\le i\le N_n}e_{i, n}$ is the largest multiplicity of a root of $P_n[f, g](x)$ (recall that finite fields are perfect fields);
\item $e_{f, g}(n):=\min\limits_{1\le i\le N_n}e_{i, n}$ is the smallest multiplicity of a root of $P_n[f, g](x)$;
\item $\Delta_{f, g}(n):=\deg(p_{1, n}(x)\cdots p_{N_n, n}(x))$ is the degree of the squarefree part of $P_n[f, g](x)$;
\item $M_{f, g}(n):=\max\limits_{1\le i\le N_n} \deg(p_{i, n}(x))$ is the largest degree of an irreducible factor of $P_n[f, g](x)$ over $\F_q$;
\item $m_{f, g}(n):=\min\limits_{1\le i\le N_n} \deg(p_{i, n}(x))$ is the smallest degree of an irreducible factor of $P_n[f, g](x)$ over $\F_q$;
\item $N_{f, g}(n):=N_n$ is the number of distinct irreducible factors of $P_n[f, g](x)$ over $\F_q$;
\item $A_{f, g}(n):=\frac{\Delta_{f, g}(n)}{N_{f, g}(n)}$ is the average degree of the distinct irreducible factors of $P_n[f, g](x)$ over $\F_q$.
\end{enumerate}
\end{definition}

If $f$ is a constant or $g$ has degree at most one, the arithmetic functions above are all constant. For this reason, we may restrict ourselves to the case $\deg(f)\ge 1$ and $\deg(g)\ge 2$. Our results have implications in the study of the asymptotic behaviour of the functions
in Definition~\ref{def:main} and, in particular, they extend some results in Theorem~\ref{eq:shp}. Most notably, in contrast to the lower bound $\Delta_n(f)\gg n^{1-\varepsilon}$ given in Theorem~\ref{eq:shp} we show that, up to genuine exceptions, $\Delta_n(f)$ has exponential growth. Furthermore, the condition $\gcd(\deg(f), q)=1$ is replaced by a less restrictive one. For more details, see Corollary~\ref{cor:MAIN}.

The structure of the paper is given as follows. In Section 2, we present the main results and provide useful remarks. Section 3 is devoted to provide some definitions and background material that is used along the way. In Sections 4 and 5 we prove our main results. Finally, in Section 6, we present conclusions and some open problems.

\section{Main results}
In this section, we provide the main results of this paper. We start with the following definition.
\begin{definition}
Let $p$ be the characteristic of $\F_q$ and consider $f, g\in \F_q[x]$ such that $\deg(f)=k\ge 1$ and $\deg(g)=D\ge 2$. 
\begin{enumerate}[(a)]
\item the \emph{$p$-reduction} of $g$ is the unique polynomial $G\in \F_q[x]$ such that $g(x)=G(x)^{p^h}$ for some $h\ge 0$ and the formal derivative $G'(x)$ of $G(x)$ is not the zero polynomial; 
\item the pair $(f, g)$ is \emph{$p$-critical} if the $p$-reduction of $g(x)$ is a linear polynomial, i.e., $g(x)=ax^{p^h}+b$ for some $h\ge 0$ and $a, b\in \F_q$;
\item the pair $(f, g)$ is \emph{critical} if there exist elements $\alpha, \beta$ and $\gamma$ in $\F_q$ such that $f(x)=\beta(x-\alpha)^k$ and $g(x)=\gamma(x-\alpha)^D+\alpha$.
\end{enumerate}
\end{definition}

A trivial upper bound for all the arithmetic functions in Definition~\ref{def:main} is $kD^n$, where $k=\deg(f)$ and $D=\deg(g)$.
It is a routine exercise to check that, if $(f, g)$ is critical or $p$-critical, $E_{f, g}(n)=E_{f, g}(0)\cdot D^n$, $e_{f, g}(n)=e_{f, g}(0)\cdot D^n$ and all the other arithmetic functions are constant, equal to their values at $n=0$. So, from now and on, we take the following assumption:
\begin{center}\fbox{The pair $(f, g)$ is neither critical nor $p$-critical.}\end{center}
The main results of this paper provide finner informations on the growth the arithmetic functions related to $P_n[f, g](x)$ and can be stated as follows.

\begin{theorem}\label{thm:main1}
Let $f, g\in \F_q[x]$ be polynomials of degrees $k\ge 1$ and $D\ge 2$, respectively, such that the pair $(f, g)$ is neither critical nor $p$-critical. Let $G\in \F_q[x]$ be the $p$-reduction of $g$, write $g(x)=G(x)^{p^h}$ and set $d=\deg(G)>1$. Then the following hold:
\begin{enumerate}[(i)]
\item $e_{f, g}(n)\ge p^{nh}$ for $n\ge 0$;
\item $E_{f, g}(n)\le \sqrt{\frac{D}{D-1}}\cdot \kappa_g^n\le \sqrt{2}\cdot \kappa_g^n$ for any $n\ge 0$, where $\kappa_g=\sqrt{D(D-1)}<D$. In particular, this bound is better than the trivial bound $E_{f, g}(n)\le k\cdot D^n$.
\item there exists a constant $C_{f, g}>0$ such that $\Delta_{f, g}(n)\ge C_{f, g}\cdot d^n$ for $n\ge 0$.
\end{enumerate}
In particular $\Delta_{f, g}(n)\approx d^n$ and so the function $\Delta_{f, g}(n)$ grows exponentially.
\end{theorem}

\begin{remark}\label{remark:main}
We observe that, under the conditions of Theorem~\ref{thm:main1}, the following holds
$$d^n\ll \Delta_{f, g}(n)= A_{f, g}(n)\cdot N_{f, g}(n).$$
In particular, there exists $c>0$ such that, for any $n\ge 0$, either $N_{f, g}(n)\ge c\cdot d^{n/2}$ or $A_{f, g}(n)\ge c\cdot d^{n/2}$.
\end{remark}

Our second result concerns the degree of the irreducible factors of $P_n[f, g](x)$.

\begin{theorem}\label{thm:main2}
Let $f, g\in \F_q[x]$ be polynomials of degrees $k\ge 1$ and $D\ge 2$, respectively, such that the pair $(f, g)$ is neither critical nor $p$-critical. Then $M_{f, g}(n)\gg n$. Moreover, the following are equivalent:

\begin{enumerate}[(i)]
\item there exists a constant $c>0$ such that $m_{f, g}(n)\le c$ for any $n\ge 0$;
\item the sequence $m_{f, g}(n)$ is eventually constant;
\item $f$ has a root $\beta\in \overline{\F}_q$ such that $g^{(i)}(\beta)=\beta$ for some $i\ge 1$.
\end{enumerate}
\end{theorem}

Following the notation of Theorem~\ref{thm:shpa} we observe that, for a generic polynomial $f\in \F_q[x]$ of degree at least one, $A_{f, f}(n)\le M_{f, f}(n)=D_{n+1}(f)$, $N_{f, f}(n)=r_{n+1}(f)$ and $\Delta_{f, f}(n)=\Delta_{n+1}(f)$ for any $n\ge 0$. In addition, the pair $(f, f)$ is critical or $p$-critical according to $f(x)$ is of the form $ax^k$ or of the form $ax^{p^h}+b$, respectively. Therefore, if $\deg(f)>1$ is relatively prime with $q$, the pair $(f, f)$ cannot be $p$-critical and it is critical if and only if $f$ is of the form $ax^k$. In particular, Theorems~\ref{thm:main1} and~\ref{thm:main2} and Remark~\ref{remark:main} provide extensions and significant improvements on the results in Theorem~\ref{thm:shpa}.

\begin{corollary}\label{cor:MAIN}
Suppose that $f\in \F_q[x]$ has degree $D>1$ and is not of the form $ax^k$ or of the form $ax^{p^h}+b$. Let $d>1$ be the degree of the $p$-reduction of $f$. Then the following hold:
\begin{enumerate}[(i)]
\item $\Delta_n(f)\approx d^n$ and so $\Delta_n(f)$ grows exponentially;
\item $D_n(f)\gg n$;
\item there exists $C_f>0$ such that, for any $n\ge 0$, either $r_{n}(f)\ge C_f\cdot d^{n/2}$ or $D_{n}(f)\ge C_f\cdot d^{n/2}$.
\end{enumerate}
\end{corollary}

We emphasize that the results of Theorems~\ref{thm:main1} and~\ref{thm:main2} and Remark~\ref{remark:main} are ``roughly sharp" in the sense that the bounds on the ``growth-type" of the functions provided in such theorem can be reached. More specifically, we have the following result.

\begin{theorem}\label{thm:main3}
Let $f\in \F_q[x]$ be any polynomial of degree at least one. Then the following hold.

\begin{enumerate}[(i)]
\item If $f$ is not of the form $ax^k$, then for any nonnegative integer $t$, there exist infinitely many polynomials $g$ such that $N_{f, g}(n)\approx n^t$ and $M_{f, g}(n)\approx \deg(g)^n$. 
\item There exist infinitely many polynomials $g$ such that $M_{f, g}(n)\approx n$.
\item If $q>2$ is a prime power and $f$ writes as $f(x)=x^k\cdot F(x)$, where $F(x)=x^m+\sum_{i=0}^{m-1}a_ix^i$ with $m\ge 1$ and $a_0\ne 0, 1$, then there exists a polynomial $g\in \F_q[x]$ such that
$M_{f, g}(n)\gg 2^n$ and $N_{f, g}(n)\gg 2^n$.
\end{enumerate}
\end{theorem}

We comment that the proof of Theorem~\ref{thm:main3} is constructive and, in particular, item (i), (ii) and (iii) are proved considering $g$ a monomial, a \emph{$q$-linearized polynomial} (i.e., a polynomial of the form $\sum_{i=0}^{m}a_ix^{q^i}$) and a mix of them, respectively. Moreover, for $g\in \F_q[x]$ a monomial or $q$-linearized and generic irreducible $f\in \F_q[x]$, we obtain exact implicit formulas for all the functions given in Definition~\ref{def:main}. This is done using the main results of~\cite{B55} (the monomial case) and~\cite{R18} (the linearized case). For more details, see Section 5.
\section{Preliminaries}
Throughout this paper, we fix $\F_q$ a finite field of characteristic $p$. In this section, we introduce some useful definitions and provide background material that is used along the way. We further show that Theorems~\ref{thm:main1} and~\ref{thm:main2} need to be proved only for a restricted class of pairs $(f, g)$.

\subsection{Notations and basic results} Let $\overline{\F}_q$ be the algebraic closure of $\F_q$. We usually denote by $Q$ a power of $q$ so $\F_Q$ is an extension of $\F_q$.

\begin{definition}
If $F\in \F_Q[x]$, we consider the following functions:

\begin{enumerate}[(a)]
\item $\nu(F)$ is the largest multiplicity of a root of $F$ over $\overline{\F}_q$;
\item $\nu^*(F)$ is the smallest multiplicity of a root of $F$ over $\overline{\F}_q$;
\item $M_Q(F)$ is the largest degree of an irreducible factor of $F$ over $\F_Q$;
\item $m_Q(F)$ is the smallest degree of an irreducible factor of $F$ over $\F_Q$;
\item $\Delta(F)$ is the degree of the squarefree part of $F$;
\item $N_Q(F)$ is the number of irreducible factors of $F$ over $\F_Q$.
\end{enumerate}
\end{definition}

For any integer $i$, let $\sigma_i:\overline{\F}_q\to \overline{\F}_q$ be the $i$-th power of the Frobenius automorphism, i.e., $\sigma_i(\alpha)=\alpha^{p^i}$. For simplicity,
$\sigma_i:\overline{\F}_q[x]\to \overline{\F}_q[x]$ also denotes the natural extension of $\sigma_i$ to the polynomial ring $\overline{\F}_q[x]$, i.e., for $f\in \overline{\F}_q[x]$ with $f(x)=\sum_{j=0}^ka_jx^j$, we have $\sigma_i(f)=\sum_{j=0}^ka_j^{p^i}x^j$. The following result is classical.

\begin{lemma}\label{lem:frobenius}
A polynomial $f\in \F_q[x]$ is irreducible if and only if $\sigma_i(f)\in \F_q[x]$ is irreducible for any $i\in\mathbb Z$. 
\end{lemma}

\begin{definition}
For an element $\alpha\in \overline{\F}_q$, $\deg_Q(\alpha)$ denotes the degree of the minimal polynomial $m_{\alpha, Q}$ of $\alpha$ over $\F_Q$. Equivalently, $\deg_Q(\alpha)$ is the least positive integer $s$ such that $\alpha\in \F_{Q^s}$.
\end{definition}

\begin{remark}\label{remark:degree}
If $f\in \F_Q[x]$ is irreducible and $\alpha\in \overline{\F}_Q=\overline{\F}_q$ is such that $f(\alpha)=0$, then $\deg_Q(\alpha)=\deg(f)$ and $f(x)=a\cdot m_{\alpha, Q}(x)$ for some $a\in \F_Q$.
\end{remark} 

\subsection{Factorization of composed polynomials} Here we provide a finite field generalization of the well-known Capelli's Lemma, that gives a criterion on the irreducibility of composed polynomials $f(g(x))$ with $f$ irreducible. This is done via the theory of spins of polynomials, introduced in~\cite{MM10}. We just state the results without proof since they are quite simple and are proved in Section 4 of~\cite{MM10}. If $Q=p^m$, set $\tau_{Q, j}=\sigma_{mj}$ for $j\ge 0$, i.e., $\tau_{Q,j}(\alpha)=\alpha^{Q^j}$ for any $\alpha\in \overline{\F}_q$. Of course, $\tau_j$ naturally extends to the polynomial ring $\overline{\F}_q[x]$. We start with the following definition.

\begin{definition}
For a polynomial $f\in \overline{\F}_Q[x]$, let $s=s_Q(f)$ be the least positive integer such that all the coefficients of $f$ lie in $\F_{Q^s}$. We define the \emph{spin} of $f$ over $\F_Q$ as
$$S_Q(f)=\prod_{j=0}^{s_Q(f)-1}\tau_{Q, j}(f).$$
\end{definition}
The following result provides a way of obtaining the factorization of composed polynomials via spins.

\begin{lemma}[see Lemmas 11 and 13 of~\cite{MM10}]\label{lem:spin1}
Let $g\in \F_Q[x]$ be a polynomial of degree at least one. For an element $\lambda\in \overline{\F}_Q$ with $\deg_Q(\lambda)=s$, we have the following equality
$$S_Q(g(x)-\lambda)=\prod_{j=0}^{s-1}(g(x)-\lambda^{Q^j}).$$
In addition, if $f\in\F_Q[x]$ is irreducible of degree $k$ and $\alpha\in \overline{\F}_Q$ is any of its roots, then the factorization of $f(g(x))$ over $\F_Q$ is given by
$$f(g(x))=\lambda_f\cdot \prod_{R}S_Q(R(x)),$$
where $R$ runs over all the irreducible factors of $g(x)-\alpha$ over $\F_{Q^k}$ and $\lambda_f$ is the leading coefficient of $f$.
\end{lemma}

From the previous lemma, we obtain the following corollary.

\begin{corollary}\label{cor:spin}
Let $\alpha\in \overline{\F}_Q$ be an element with $\deg_Q(\alpha)=k$ and let $g\in \F_Q[x]$ be a polynomial of degree at least one. If
$$g(x)-\alpha=p_1(x)^{e_1}\cdots p_{r}(x)^{e_r},$$
is the factorization of $g(x)-\alpha$ over $\F_{Q^k}$  then, for any $1\le i\le r$, we have that $\deg(S_Q(p_i(x)))=k\cdot \deg(p_i)$.
\end{corollary}

\begin{proof}
Let $f$ be the minimal polynomial of $\alpha$ over $\F_Q$, hence $f$ is irreducible and $\deg(f)=k$. From Lemma~\ref{lem:spin1}, we have that
\begin{equation}\label{eq:spins}
f(g(x))=\prod_{i=0}^{r}S_Q(p_i(x))^{e_i}.
\end{equation}
Since $p_i\in \F_{Q^k}[x]$, it follows that $s_Q(p_i)\le k$. In particular, from the definition, we have that $$\deg(S_Q(p_i(x)))=s_Q(p_i(x))\cdot \deg(p_i)\le k\cdot \deg(p_i).$$ Taking degrees on Eq.~\eqref{eq:spins} we obtain
$$k\cdot \deg(g)=\sum_{i=1}^re_i\cdot \deg(S_Q(p_i(x)))\le k\cdot\sum_{i=1}^re_i\cdot \deg(p_i)=k\cdot \deg(g).$$
Therefore, we necessarily have the equality $\deg(S_Q(p_i(x)))=k\cdot \deg(p_i)$.
\end{proof}

Lemma~\ref{lem:spin1} and Corollary~\ref{cor:spin} immediately give the following results.

\begin{lemma}\label{lem:prime}
Let $G, F_1, F_2\in \overline{\F}_Q[x]$ be polynomials of degree at least one. If $F_1$ and $F_2$ are relatively prime, then $F_1(G(x))$ and $F_2(G(x))$ are relatively prime.
\end{lemma}

\begin{proposition}\label{prop:identities}
Let $f\in \F_q[x]$ be an irreducible polynomial of degree $k$ and let $\alpha\in \F_{Q}$ be any of its roots, where $Q=q^k$. For any polynomial $g\in \F_q[x]$ of degree at least one and any $n\ge 0$, the following hold:

\begin{enumerate}[(i)]
\item $E_{f, g}(n)=\nu(g^{(n)}(x)-\alpha)$;
\item $e_{f, g}(n)=\nu^*(g^{(n)}(x)-\alpha)$;
\item $M_{f, g}(n)=k\cdot M_{Q}(g^{(n)}(x)-\alpha)$;
\item $m_{f, g}(n)=k\cdot m_{Q}(g^{(n)}(x)-\alpha)$;
\item $\Delta_{f, g}(n)=k\cdot \Delta(g^{(n)}(x)-\alpha)$;
\item $N_{f, g}(n)=N_{Q}(g^{(n)}(x)-\alpha)$.
\end{enumerate}
\end{proposition}

\begin{corollary}\label{cor:nondecreasing}
 For any polynomial $g\in \F_Q[x]$ of degree at least two and any $\alpha$ in an extension $\F_{Q^k}$ of $\F_Q$, the function $m_{Q^k}(g^{(n)}(x)-\alpha)$ is nondecreasing.
 \end{corollary}

The following lemma provides bounds for the values of $\nu$ and $\nu^*$ at iterated polynomials $g^{(n)}(x)-\alpha$.

\begin{lemma}\label{lem:mult-1}
Let $\alpha$ be an element in $\overline{\F}_q$, $g\in \overline{\F}_q[x]$ be a polynomial of degree at least one and $n\ge 1$. The following hold: 
\begin{equation}\label{eq:nu}
\nu(g^{(n)}(x)-\alpha)\le \max\limits_{\Gamma\in C_n(\alpha)}\prod_{\lambda \in \Gamma}\nu(g(x)-\lambda)
\end{equation}
and
\begin{equation}\label{eq:nu*}
\nu^*(g^{(n)}(x)-\alpha)\ge \min\limits_{\Gamma\in C_n(\alpha)}\prod_{\lambda \in \Gamma}\nu^*(g(x)-\lambda),
\end{equation}
where $C_n(\alpha)$ comprises the $n$-tuples $\Gamma=(\gamma_1, \ldots, \gamma_n)\in (\overline{\F}_q)^n$ such that $\gamma_1=\alpha$ and, if $n\ge 2$, $g(\gamma_i)=\gamma_{i-1}$ for $2\le i\le n$.
\end{lemma}

\begin{proof}
We only consider the Ineq.~\eqref{eq:nu} since the proof of Ineq.~\eqref{eq:nu*} is entirely similar. We proceed by induction on $n$. The case $n=1$ is trivial. Suppose that the result holds for an integer $N\ge 1$ and let $n=N+1$. If $g(x)-\alpha=\prod_{i=1}^s(x-\delta_i)^{e_i}$, where the $\delta_i$ are pairwise distinct, we have that 
$$\nu(g^{(n)}(x)-\alpha)=\max _{1\le i\le s}e_i\cdot \nu(g^{(N)}(x)-\delta_i).$$
From induction hypothesis and the trivial inequality $e_i\le \nu(g(x)-\alpha)$, we have that
$$\nu(g^{(n)}(x)-\alpha)\le  \nu(g(x)-\alpha)\cdot \max_{\Gamma\in C_N(\delta_i), 1\le i\le s}\prod_{\lambda \in \Gamma}\nu(g(x)-\lambda).$$
It follows by the definition that the elements of $C_n(\alpha)$ are exactly the sets of the form $\{\alpha, \gamma_2, \ldots, \gamma_{n-1}\}$, where 
$\{\gamma_2, \ldots, \gamma_{n-1}\}\in C_N(\delta_i)$ for some $1\le i\le s$. In particular,
$$ \nu(g(x)-\alpha)\cdot \max_{\Gamma\in C_N(\delta_i), 1\le i\le s}\prod_{\lambda \in \Gamma}\nu(g(x)-\lambda)=\max\limits_{\Gamma\in C_n(\alpha)}\prod_{\lambda \in \Gamma}\nu(g(x)-\lambda),$$
and the result follows.
\end{proof}

\subsection{A reduction of Theorems~\ref{thm:main1} and~\ref{thm:main2}}
Here we show that Theorems~\ref{thm:main1} and~\ref{thm:main2} need only to be proved for pairs $(f, g)$ such that $f$ is irreducible. 

\begin{lemma}\label{lem:reduction-1}
 Let $f, g\in \F_q[x]$ be polynomials such that $\deg(f)\ge 1$ and $\deg(g)\ge 2$. Suppose that the irreducible factorization of $f$ over $\F_q$ is
$$f(x)=f_1(x)^{e_1}\cdots f_s(x)^{e_s}.$$ 
Then the pair $(f, g)$ is critical (resp. $p$-critical) if and only if $(f_i, g)$ is critical (resp. $p$-critical) for any $1\le i\le n$. In addition, for any $n\ge 0$, the following hold:
\begin{enumerate}[(a)]
\item $E_{f, g}(n)=\max\limits_{1\le i\le s} \{e_i\cdot E_{f_i, g}(n)\}$ if $n\ge 1$;
\item $e_{f, g}(n)=\min\limits_{1\le i\le s}\{e_i\cdot e_{f_i, g}(n)\}$ if $n\ge 1$;
\item $M_{f, g}(n)=\max\limits_{1\le i\le s} \{M_{f_i, g}(n)\}$;
\item $m_{f, g}(n)=\min\limits_{1\le i\le s} \{m_{f_i, g}(n)\}$;
 \item $\Delta_{f, g}(n)= \sum_{1\le i\le s} \Delta_{f_i, g}(n)$;
\item $N_{f, g}(n)= \sum_{1\le i\le s} N_{f_i, g}(n)$.
 \end{enumerate}
 In particular, if Theorems~\ref{thm:main1} and~\ref{thm:main2} hold for the pairs $(f, g)$ with $f$ irreducible, then they hold for any pair $(f, g)$.
 \end{lemma}

 \begin{proof}
 The first statement follows directly from the definition of critical and $p$-critical pairs. Also, items (a--f) follow from Lemma~\ref{lem:prime} and the fact that 
 $$P_n[f, g](x)=\prod_{i=1}^sP_n[f_i, g](x)^{e_i},\; n\ge 0.$$
 For the last statement we observe that, for a fixed $f$, the number $s$ is bounded by $\deg(f)$, that does not depend on $n$. In particular, from the bounds in the items (a--f) above, Theorems~\ref{thm:main1} and~\ref{thm:main2} hold for the pair $(f, g)$ whenever they hold for each pair $(f_i, g)$ with $1\le i\le s$. 
 \end{proof}
\section{Proof of Theorems~\ref{thm:main1} and~\ref{thm:main2}}
In this section we provide the proof of Theorems~\ref{thm:main1} and~\ref{thm:main2}, that is divided in many parts. We observe that, from Lemma~\ref{lem:reduction-1}, we only need to consider pairs $(f, g)$ such that $f$ is irreducible. So we take the following assumption:
\begin{center}\fbox{The pair $(f, g)$ is neither critical nor $p$-critical and $f$ is irreducible.}\end{center}
Before we proceed to the proof of Theorems~\ref{thm:main1} and~\ref{thm:main2}, we comment on the main ideas that are employed. If $f\in \F_q[x]$ has degree $k$ and is irreducible, Proposition~\ref{prop:identities} entails that the arithmetic functions associated to $f(g^{(n)}(x))$ depend on the factorization of the polynomial $g^{(n)}(x)-\alpha$ over $\F_{q^k}$, where $\alpha$ is any root of $f$. So we focus on the factorization of polynomials of the form $g^{(n)}(x)-\alpha$ in our statements and explain how they imply the main results.
When studying the polynomials of the form $g^{(n)}(x)-\alpha$, we will frequently look at the {\em reversed $g$-orbit} of $\alpha$, that is, the set $\{\beta\in \overline{\F}_q\,|\, g^{(i)}(\beta)=\alpha\; \text{for some}\; i\ge 0\}$.

We start with the following definition.
\begin{definition}
For a polynomial $g\in \F_q[x]$, an element $\lambda\in \overline{\F}_q$ is said to be \emph{$g$-periodic} if there exists a positive integer $i$ such that $g^{(i)}(\lambda)=\lambda$.
\end{definition}
We have the following proposition.

\begin{proposition}\label{prop:main1}
Let $g\in \F_q[x]$ be a polynomial of degree $D\ge 2$ and $\alpha\in \overline{\F}_q$. Suppose that $G\in \F_q[x]$ is the $p$-reduction of $g$ with $g(x)=G(x)^{p^h}$, where $h\ge 0$. If $g(x)$ is not of the form $a(x-\alpha)^D+\alpha$ or of the form $ax^{p^h}+b$ then, for any $n\ge 0$, the following holds: $$p^{nh}\le \nu^*(g^{(n)}(x)-\alpha)\le \nu(g^{(n)}(x)-\alpha)\le \sqrt{\frac{D}{D-1}}\cdot \kappa_g^n\le \sqrt{2}\kappa_g^n\, ,$$
where $\kappa_g=\sqrt{D(D-1)}$.
\end{proposition}
\begin{proof}
The result is trivial for $n=0$ and so we suppose that $n\ge 1$. We observe that, for any $\lambda\in \overline{\F}_q$, $\nu(g(x)-\lambda)=p^h\cdot \nu(G(x)-\sigma_{-h}(\lambda))$ and  $\nu^*(g(x)-\lambda)=p^h\cdot \nu^*(G(x)-\sigma_{-h}(\lambda))$. In particular, the inequality
$$p^{nh}\le \nu^*(g^{(n)}(x)-\alpha)\le \nu(g^{(n)}(x)-\alpha),$$
follows trivially from Ineq.~\eqref{eq:nu*}. Following the notation of Lemma~\ref{lem:mult-1}, $C_n(\alpha)$ comprises the $n$-tuples $$\Gamma=(\gamma_1,\ldots, \gamma_n)\in (\overline{\F}_q)^n,$$ such that $\gamma_1=\alpha$ and, if $n\ge 2$, $g(\gamma_i)=\gamma_{i-1}$ for $2\le i\le n$. From Ineq.~\eqref{eq:nu}, we have that
\begin{equation}\label{eq:nuu}\nu(g^{(n)}(x)-\alpha)\le \max\limits_{\Gamma\in C_n(\alpha)}\left\{\prod_{\lambda \in \Gamma}\nu(g(x)-\lambda)\right\}.\end{equation}
We observe that $\nu(g(x)-\lambda))\le D=\deg(g)$, for any $\lambda\in \overline{\F}_q$. 

\begin{enumerate}[]
\item {\bf Claim 1.} \emph{The equality $\nu(g(x)-\gamma)=D$ holds for at most one element $\gamma\in \overline{\F}_q$}.
\item {\it Proof of Claim 1.} 
Since $g(x)$ is not of the form $ax^{p^h}+b$ and $G\in \F_q[x]$ is the $p$-reduction of $g$, we have that $G$ has degree $d:=\frac{D}{p^h}\ge 2$ and $G'(x)$ is not the zero polynomial. Since $g(x)=G(x)^{p^h}$, it suffices to prove that $\nu(G(x)-\delta)=d$ for at most one element $\delta\in \overline{\F}_q$. However, if $\nu(G(x)-\delta)=\nu(G(x)-\delta_0)=d$ with $\delta\ne \delta_0$, it follows that the polynomial $G'(x)$ vanishes at two distinct elements of $\overline{\F}_q$ with multiplicity at least $d-1$. Since $G'(x)$ is not the zero polynomial, it follows that $\deg(G'(x))\ge 2(d-1)$. However, $\deg(G')\le \deg(G)-1=d-1$ and so $2(d-1)\le d-1$, a contradiction since $d\ge 2$.
\end{enumerate}
If it does not exist $\gamma\in \overline{\F}_q$ such that $\nu(g(x)-\gamma)=D$, Ineq.~\eqref{eq:nuu} yields $$\nu(g^{(n)}(x)-\alpha)\le (D-1)^n< \sqrt{\frac{D}{D-1}}\cdot \kappa_g^n.$$ On the contrary, let $B\in \overline{\F}_q$ be such that $\nu(g(x)-B)=D$. From Claim 1, $\nu(g(x)-\gamma)\le D-1$ whenever $\gamma\ne B$. For any $\Gamma=(\gamma_1, \ldots, \gamma_n)\in C_n(\alpha)$ and $\gamma\in \overline{\F}_q$, let $e(\gamma, \Gamma)$ be the number of indexes $1\le i\le n$ for which $\gamma=\gamma_i$. In particular, Ineq.~\eqref{eq:nuu} yields 
\begin{equation}\label{eq:nuuu}
\nu(g^{(n)}(x)-\alpha)\le \max_{\Gamma\in C_n(\alpha)}\{D^{e(B, \Gamma)}\cdot (D-1)^{n-e(B, \Gamma)}\}.
\end{equation}
Therefore, if $e(B, \Gamma)\le 1$ for any $n\ge 1$ and any $\Gamma\in C_n(\alpha)$, Ineq.~\eqref{eq:nuuu} yields $$\nu(g^{(n)}(x)-\alpha)\le D(D-1)^{n-1}< \sqrt{\frac{D}{D-1}}\cdot \kappa_{g}^n,$$ as desired. Otherwise, there exist $N\ge 2$ and $\Gamma_0\in C_N(\alpha)$ such that $B\in \Gamma_0$ and $e(B, \Gamma_0)\ge 2$. If $B=\gamma_i=\gamma_j\in \Gamma_0$ with $1\le i<j\le N$, then $g^{(j-i)}(B)=B$ and $g^{(i)}(B)=\alpha$ and so $B, \alpha$ are $g$-periodic and lie in the same orbit. Additionally, since $\nu(g(x)-B)=D$, there exist $a, b\in \overline{\F}_q$ such that $g(x)=a(x-b)^D+B$. Let $M$ be the least positive integer such that $g^{(M)}(B)=B$. 

\begin{enumerate}[]
\item {\bf Claim 2.} {\em For any $n\ge 1$ and any $\Gamma\in C_{n}(\alpha)$, we have that $$e(B, \Gamma)\le \frac{(n-1)}{M}+1.$$}
\item {\it Proof of Claim 2.} Set $s=e(B, \Gamma)$; if $s=0, 1$ the result is trivial. Otherwise, $s\ge 2$ and following the proof that $B$ is periodic, we see that there exist integers $0\le i_1<\cdots<i_s\le n-1$ such that $g^{(i_l-i_{l-1})}(B)=B$ for any $2\le l\le s$. In particular, $i_{l}-i_{l-1}\ge M$ for any $2\le l\le s$ and so 
$$n-1\ge i_{s}-i_1=\sum_{2\le l\le s}(i_{l}-i_{l-1})\ge (s-1)M,$$ proving the claim.
\end{enumerate}

From $e(B, \Gamma)\le (n-1)/M+1$ and Ineq.~\eqref{eq:nuuu} we have that, for any $n\ge 1$, the following holds
$$\nu(g^{(n)}(x)-\alpha)\le D\cdot D^{(n-1)/M}(D-1)^{(n-1)(M-1)/M}\le \sqrt{\frac{D}{D-1}}\cdot \kappa_g^{n},$$
provided that $M\ge 2$. In particular, we only need to consider the case $M=1$. However, in this case, $B=g(B)=a(B-b)^D+B$ and so $b=B$ and $g(x)=a(x-B)^D+B$. Additionally, since $\alpha$ is $g$-periodic and lies in the same orbit of $B$, we have that $\alpha=B$. In other words, $g(x)=a(x-\alpha)^D+\alpha$, contradicting our hypothesis.
\end{proof}

We observe that, from Proposition~\ref{prop:main1}, items (i) and (ii) of Theorem~\ref{thm:main1} are easily deduced. In fact, from Lemma~\ref{lem:reduction-1}, we only need prove them for $f$ an irreducible polynomial. If $f$ has degree $k$ and $\alpha\in \F_{q^k}$ is any of its roots, Proposition~\ref{prop:identities} entails that $e_{f, g}(n)=\nu^*(g^{(n)}(x)-\alpha)$ and $E_{f, g}(n)=\nu(g^{(n)}(x)-\alpha)$. Moreover, the assumption in Theorem~\ref{thm:main1} that the pair $(f, g)$ is neither critical nor $p$-critical directly implies that $g$ and $\alpha$ satisfy the conditions in Proposition~\ref{prop:main1}.

We proceed to the proof of item (iii) of Theorem~\ref{thm:main1} and the comment thereafter. The key idea is to prove that, if $g$ is not of the form $ax^{p^h}+b$ or $\beta(x-\alpha)^D+\alpha$, there exists an element $\gamma$ in the reversed $g$-orbit $$\{\beta\in \overline{\F}_q\,|\, g^{(i)}(\beta)=\alpha\; \text{for some}\; i\ge 0\}$$ of $\alpha$ such that $\Delta(g^{(n)}(x)-\gamma)\ge \deg(G)^n$ for any $n\ge 0$, where $G$ is the $p$-reduction of $g$. If $g^{(j)}(\gamma)=\alpha$, the polynomial $g^{(n)}(x)-\gamma$ divides $g^{(n+j)}(x)-\alpha$ for any $n\ge 0$. The latter implies that $\Delta(g^{(n)}(x)-\alpha)\gg \deg(G)^n$. These observations are compiled in the following lemma.

\begin{lemma}\label{lem:delta}
Let $\alpha\in\overline{\F}_q$ and let $g\in \F_q[x]$ be a polynomial of degree $D\ge 2$ such that $g$ is not of the form $ax^{p^h}+b$ or $a(x-\alpha)^D+\alpha$. Let $G\in \F_q[x]$ be the $p$-reduction of $g$, $d=\deg(G)>1$ and write $D=d\cdot p^h$ with $h\ge 0$. Then there exists a positive integer $i$ and an element $\gamma\in \overline{\F}_q$ such that $g^{(i)}(\gamma)=\alpha$ and $\nu(g^{(n)}(x)-\gamma)=p^{nh}$ for any $n\ge 0$. In particular, there exists a constant $C_{\alpha, g}>0$ such that \begin{equation}\label{eq:delta}\Delta(g^{(n)}(x)-\alpha)\ge C_{\alpha, g}\cdot d^n,\end{equation} 
for any $n\ge 0$ and so $\Delta(g^{(n)}(x)-\alpha)\approx d^n$.
\end{lemma}

\begin{proof}
First, we claim that there exists a positive integer $j=j(\alpha)$ and an element $\lambda\in \overline{\F}_q$ such that $\lambda$ is not $g$-periodic and $g^{(j)}(\lambda)=\alpha$. Proceeding by contradiction, suppose that all the roots of $g(x)-\alpha$ are $g$-periodic and, if $\eta$ is any root of $g(x)-\alpha$, all the roots of $g(x)-\eta$ are also $g$-periodic. We observe that, for any $\beta\in \overline{\F}_q$, at most one root of  the polynomial $g(x)-\beta$ is $g$-periodic. In particular, the latter implies that $g(x)-\alpha=a_1(x-\eta)^D$ and $g(x)-\eta=a_2(x-\eta_0)^D$ for some $\eta_0\in \overline{\F}_q$. In other words, $\nu(g(x)-\alpha)=\nu(g(x)-\eta)=D=\deg(g)$. From Claim~1 in Proposition~\ref{prop:main1}, we conclude that $\alpha=\eta$ and so $g(x)$ is of the form $a(x-\alpha)^D+\alpha$, contradicting our hypothesis. 

Let $\lambda\in \overline{\F}_q$ and $j>0$ be such that $g^{(j)}(\lambda)=\alpha$ and $\lambda$ is not $g$-periodic. 
Since $G$ is such that $G'(x)$ is not the zero polynomial, it follows that the set of roots of $G'(x)$ is finite. It is direct to verify that, for any $m\ge 0$, no root of $g^{(m)}(x)-\lambda$ is $g$-periodic. In particular, the polynomials $g^{(m)}(x)-\lambda$ with $m\ge 0$ are pairwise relatively prime. Therefore, there exists a positive integer $M=M(\lambda)$ such that the polynomial $g^{(k)}(x)-\lambda$ is relatively prime with $G'(x)$ for any $k\ge M$. Let $\gamma$ be any root of $g^{(M)}(x)-\lambda$.

\begin{enumerate}[]
\item {\bf Claim 1.} {\it For any $n\ge 0$, $\nu(g^{(n)}(x)-\gamma)=p^{nh}$}. 
\item {\it Proof of Claim 1.} From Proposition~\ref{prop:main1} and Ineq.~\eqref{eq:nu}, we have that
\begin{align*}p^{nh} &\le \nu(g^{(n)}(x)-\gamma)\le \max\limits_{\Gamma\in C_n(\gamma)}\left\{\prod_{\beta \in \Gamma}\nu(g(x)-\beta)\right\}\\ {} & =p^{nh}\cdot \max\limits_{\Gamma\in C_n(\gamma)}\left\{\prod_{\beta \in \Gamma}\nu(G(x)-\sigma_{-h}(\beta))\right\},\end{align*}
where $C_n(\gamma)$ is as in Lemma~\ref{lem:mult-1}. We observe that, if there exists $n\ge 0$, $\Gamma\in C_n(\gamma)$ and $\beta\in \Gamma$ such that $\nu(G(x)-\sigma_{-h}(\beta))>1$, then there exists an element $\beta_0\in \overline{\F}_q$ such that $G(\beta_0)=\sigma_{-h}(\beta)$ and $G'(\beta_0)=0$. However, $g(\beta_0)=\beta\in C_n(\gamma)$ and so $\beta_0$ is a root of $g^{(t)}(x)-\gamma$ for some $0\le t\le n+1$. The latter implies that $\beta_0$ is a root of $g^{(t+M)}(x)-\lambda$, which is relatively prime with $G'(x)$ since $t+M\ge M$. But this is contradiction with the equality $G'(\beta_0)=0$. In particular, for any $n\ge 0$, we have that 
$$\max\limits_{\Gamma\in C_n(\gamma)}\left\{\prod_{\beta \in \Gamma}\nu(G(x)-\sigma_{-h}(\beta))\right\}=1,$$
proving the claim.
\end{enumerate}

Therefore, for $i=j+M$, $g^{(i)}(\gamma)=\alpha$ and $\nu(g^{(n)}(x)-\gamma)=p^{nh}$ for any $n\ge 0$. We observe that, since $g^{(i)}(\gamma)=\alpha$, $g^{(n)}(x)-\gamma$ divides $g^{(n+i)}(x)-\alpha$ and so we have that 
$$\Delta(g^{(n)}(x)-\alpha)\ge \Delta(g^{(\chi_i(n))}(x)-\gamma),$$
where $\chi_i(n)=\max\{n-i, 0\}$. In addition, we have the trivial inequality
$$\Delta(g^{(m)}(x)-\gamma)\ge \frac{\deg(g^{(m)}(x)-\gamma)}{\nu(g^{(m)}(x)-\gamma)}=d^m,$$
for any $m\ge 0$ and so Ineq.~\eqref{eq:delta} holds with $C_{\alpha, g}=\frac{1}{d^i}$. To finish the proof we observe that, from Proposition~\ref{prop:main1}, $\nu^*(g^{(n)}(x)-\alpha)\ge p^{nh}$ and so 
$$\Delta(g^{(n)}(x)-\alpha)\le \frac{\deg(g^{(n)}(x)-\alpha)}{\nu^*(g^{(n)}(x)-\alpha)}\le d^n.$$
In particular, $\Delta(g^{(n)}(x)-\alpha)\approx d^n$.
\end{proof}

We observe that Lemma~\ref{lem:delta} immediately implies item (iii) of Theorem~\ref{thm:main1} and the comment thereafter. In fact,  from Lemma~\ref{lem:reduction-1}, we only need to consider the pairs $(f, g)$ as in Theorem~\ref{thm:main1} with $f\in \F_q[x]$ irreducible. Moreover, if $f$ is irreducible of degree $k$ and $\alpha\in \F_{q^k}$ is any of its roots, Proposition~\ref{prop:identities} entails that $\Delta_{f, g}(n)=k\cdot \Delta(g^{(n)}(x)-\alpha)$. Since the pair $(f, g)$ is neither critical nor $p$-critical, we have that $\alpha$ and $g$ satisfy the conditions of Lemma~\ref{lem:delta} and so $\Delta(g^{(n)}(x)-\alpha)\approx d^n$, where $d$ is the degree of the $p$-reduction $G\in \F_q[x]$ of $g$.

\subsection{On the degree and number of irreducible factors of $P_n[f, g](x)$} Here we provide the proof of Theorem~\ref{thm:main2}. We start with the bound on $M_{f, g}(n)$.

\begin{lemma}\label{lem:max-degree}
Let $f, g\in \F_q[x]$ be polynomials of degrees $k\ge 1$ and $D\ge 2$, respectively, such that $f$ is irreducible and the pair $(f, g)$ is neither critical nor $p$-critical. Then $M_{f, g}(n)\gg n$.
\end{lemma}

\begin{proof}
Let $\alpha\in\overline{\F}_q$ be any root of $f$. From Proposition~\ref{prop:identities}, we only need to prove that $M_{Q}(g^{(n)}(x)-\alpha)\gg n$ for $Q=q^k$. If we set $M(n)=M_Q(g^{(n)}(x)-\alpha)$, it follows that the roots of $g^{(n)}(x)-\alpha$ lie in $\mathcal C_{M(n)}=\cup_{1\le i\le M(n)}\F_{Q^i}$. Since the pair $(f, g)$ is neither critical or $p$-critical, $g$ is not of the form $\beta(x-\alpha)^D+\alpha$ or of the form $ax^{p^h}+b$. Therefore, from Lemma~\ref{lem:delta}, it follows that $\Delta(g^{(n)}(x)-\alpha)\ge c\cdot d^n$ for some $c>0$ (that does not depend on $n$), where $d>1$ is the degree of the $p$-reduction of $g$. In particular, we have that
$$c\cdot d^n\le \Delta(g^{(n)}(x)-\alpha)\le |\mathcal C_{M(n)}|\le Q^{M(n)}+\ldots+Q<2Q^{M(n)}.$$
Thus, $M(n)> \frac{n\log (d)+\log c-\log 2}{\log Q}\gg n,$
since $Q, d>1$.
\end{proof}

The following result proves the remaining statement in Theorem~\ref{thm:main2}, concerning the arithmetic function $m_{f, g}(n)$.

\begin{lemma}\label{lem:min-1}
Let $f\in \F_q[x]$ be an irreducible polynomial of degree $k\ge 1$ and let $\alpha\in \F_{q^k}$ be any of its roots. If $g$ is any polynomial of degree $D\ge 1$, then the following are equivalent:

\begin{enumerate}[(i)]
\item there exists a constant $c>0$ such that $m_{f, g}(n)\le c$ for any $n\ge 0$;
\item the sequence $\{m_{f, g}(n)\}_{n
\ge 0}$ is eventually constant;
\item $\alpha$ is $g$-periodic;
\item $m_{f, g}(n)=k$ for any $n\ge 0$.
\end{enumerate}
\end{lemma}

\begin{proof}
From Proposition~\ref{prop:identities} and Corollary~\ref{cor:nondecreasing}, $m_{f, g}(n)$ is nondecreasing and so (i) implies (ii). To see that (ii) implies (iii), suppose that $m_{f, g}(n)=R$ for some positive integer $R$ and any $n\ge n_0$, where $n_0>0$. Since there exist a finite number of irreducible polynomials of degree $R$ over $\F_q$, there exist positive integers $M> N$ and an irreducible polynomial $h\in \F_q[x]$ of degree $R$ such that $h(x)$ divides $P_N[f, g](x)$ and $P_M[f, g](x)$. The latter implies that there exists a root $\gamma\in\overline{\F}_q$ of $h$ that is a root of $g^{(N)}(x)-\alpha^{q^s}$ and $g^{(M)}(x)-\alpha^{q^t}$ for some $1\le s, t\le k$. Therefore, $g^{(N)}(\gamma)=\alpha^{q^s}$ and $g^{(M)}(\gamma)=\alpha^{q^t}$ and so $g^{(M-N)}(\alpha^{q^s})=\alpha^{q^t}$. Since $g\in \F_q[x]$, we have that $g(x^q)=g(x)^{q}$ and so $g^{(M-N)}(\alpha)=\alpha^{q^{t-s}}$. The latter implies that $g^{(k(M-n))}(\alpha)=\alpha^{q^{k(t-s)}}=\alpha$, i.e., $\alpha$ is $g$-periodic. To prove that (iii) implies (iv) we observe that $m_{f, g}(0)=\deg(f)=k$ and so $m_{f, g}(n)\ge k$ for any $n\ge 0$, since $m_{f, g}(n)$ is nondecreasing. However, if $\alpha$ is $g$-periodic, $g^{(i)}(\alpha)=\alpha$ for some positive integer $i$ and so $\alpha$ is a root of $g^{(mi)}(x)-\alpha$ for any $m\ge 0$. In particular, $m_{f, g}(mi)=k$ for any $m\ge 0$ and again, since $m_{f, g}(n)$ is nondecreasing, we have that $m_{f, g}(n)=k$ for any $n\ge 0$. Finally, (iv) trivially implies (i). 
\end{proof}

\section{The cases $g$ monomial and $q$-linearized}
In this section we explicitly compute the arithmetic functions associated to the factorization of polynomials $P_n[f, g](x)$ when $g$ is a monomial or a $q$-linearized polynomial, i.e., a polynomial of the form $\sum_{i=0}^ma_ix^{q^i}$. In particular, we provide a constructive proof of Theorem~\ref{thm:main3}. We are mainly interested in the case in which $f$ is irreducible, since the general case can be obtained from the identities of Lemma~\ref{lem:reduction-1}. In the following lemma, we provide a reduction to the case in which $g'(x)$ is not the zero polynomial.

\begin{lemma}
Suppose that $g(x)=x^m$ or $g(x)=\sum_{i=0}^ma_ix^{q^i}\in \F_q[x]$, let $G\in \F_q[x]$ be the $p$-reduction of $g$ and write $g(x)=G(x)^{p^h}$. Then, for any $n\ge 0$, $e_{f, g}(n)=p^{nh}\cdot e_{f, G}(n)$, $E_{f, g}(n)=p^{nh}\cdot E_{f, G}(n)$ and all the other arithmetic functions in Definition~\ref{def:main} coincide at every $n\ge 0$ when evaluated for the pairs $(f, g)$ and $(f, G)$.
\end{lemma}

\begin{proof}
We observe that, if $g(x)=\sum_{i=0}^ma_ix^{q^i}$, then $p^h$ must be a power of $q$ and so $G(x^{p^h})=G(x)^{p^h}$; the last equality trivially holds if $g$ is a monomial. In particular, $\sigma_h(G)=G$, i.e., $G\in \F_{p^h}[x]$. Therefore, it follows by induction on $n\ge 0$ that
$$P_n[f, g](x)=f\left(G^{(n)}(x)^{p^{nh}}\right)=[\sigma_{-nh}(f)(G^{(n)}(x))]^{p^{nh}}=\sigma_{-nh}\left(P_n[f, G](x)\right)^{p^{nh}}.$$
Fix $n\ge 0$ and let 
$$P_n[f, G](x)=h_1(x)^{e_1}\cdots h_r(x)^{e_r},$$
be the factorization of $P_n[f, G]$ into irreducible polynomials over $\F_q$. From Lemma~\ref{lem:frobenius}, it follows that the factorization of $\sigma_{-nh}\left(P_n[f, G](x)^{p^{{nh}}}\right)=P_n[f, g](x)$ into irreducible polynomials over $\F_q$ is
$$P_n[f, g](x)=H_1(x)^{p^{nh}\cdot e_1}\cdots H_r(x)^{p^{nh}\cdot e_r},$$
where $H_i(x)=\sigma_{-nh}(h_i(x))$ for any $1\le i\le r$.
\end{proof}

We emphasize that the previous lemma holds for any polynomial $g\in\F_q[x]$ with $p$-reduction $G$ such that $g(x)=G(x)^{p^h}$ and $G\in \F_{p^h}[x]$. So, for the rest of this section, we may assume that $g(x)$ is a monomial or a $q$-linearized polynomial such that $g'(x)$ is not the zero polynomial.

\subsection{The monomial case} Here we consider $g(x)=x^D$, where $\gcd(D, p)=1$ and $f\in \F_q[x]$ an irreducible polynomial not of the form $ax$. We introduce some useful definitions.

\begin{definition}
Let $f, g\in\F_q[x]$ be polynomials such that $\gcd(f, g)=1$ and let $a, b$ be positive integers such that $\gcd(a, b)=1$.

\begin{enumerate}[(a)]
\item the order $\ou(g, f)$ of $g$ modulo $f$ is the least positive integer $s$ such that $f(x)$ divides $g(x)^s-1$;
\item if $f$ is not divisible by $x$, $\ord(f):=\ou(x, f)$ is the order of $f$;
\item the order $\ord(a, b)$ of $a$ modulo $b$ is the least positive integer $i$ such that $$a^i\equiv 1\pmod b.$$
\end{enumerate}
\end{definition}

Butler~\cite{B55} obtained the following result.

\begin{theorem}\label{thm:butler}
Let $f\in \F_q[x]$ be an irreducible polynomial of degree $k$, not of the form $ax$ such that $e=\ord(f)$. Let $m$ be a positive integer such that $\gcd(m, q)=1$ and $m=m_1m_2$, where $\gcd(m_1, e)=1$ and each prime factor of $m_2$ divides $e$. Then
\begin{enumerate}[(i)]
\item each irreducible factor of $f(x^m)$ has order of the form $Mm_2e$, where $M|m_1$;
\item if $M|m_1$, then $f(x^m)$ has exactly $\frac{nm_2\varphi(M)}{\ord(q, {Mm_2e})}$ irreducible factors of degree $\ord(q, {Mm_2e})$ with order $Mm_2e$, where $\varphi$ is the Euler's Totient function.
\end{enumerate}
\end{theorem}


From~Theorem~\ref{thm:butler}, the following result is straightforward.
\begin{corollary}\label{cor:monomial}
Let $f\in \F_q[x]$ be an irreducible polynomial of degree $k$, not of the form $ax$ such that $e=\ord(f)$. Let $D$ be a positive integer such that $\gcd(D, p)=1$ and write $D=d_1d_2$, where $\gcd(d_1, e)=1$ and each prime factor of $d_2$ divides $e$. Then, for $g(x)=x^D$ and any $n\ge 0$, the following hold:
\begin{enumerate}[(i)]
\item $E_{f, g}(n)=e_{f, g}(n)=1$;
\item $\Delta_{f, g}(n)=k\cdot D^n$;
\item $M_{f, g}(n)=\ord(q, {D^ne})$;
\item $m_{f, g}(n)=\ord(q, {d_2^ne})$;
\item $N_{f, g}(n)=\sum_{M|d_1^n}\frac{kd_2^n\varphi(M)}{\ord(q, {Md_2^ne})}.$
\end{enumerate}
\end{corollary}

In order to estimate the functions appearing in Corollary~\ref{cor:monomial}, we need to find bounds for orders $\ord(a, b)$, where the prime factors of $b$ are fixed. In this context, the prime valuation of integers is required.

\begin{definition}
For an integer $a\ne 0$ and $r$ a prime number, $\nu_r(a)$ denotes the greatest nonnegative integer $s$ such that $r^s$ divides $a$.
\end{definition}

The following lemma is a particular case of the Lifting the Exponent Lemma (LTE), a famous result in the Olympiad folklore. Its proof easily follows by induction and so we omit the details.

\begin{lemma}\label{lem:LTE} Let $r$ be a prime and $a$ a positive integer such that $r$ divides $a-1$. For any positive integer $n$, the following hold:

\begin{enumerate}[(i)]
\item if $r$ is odd, $\nu_r(a^n-1)=\nu_r(a-1)+\nu_r(n)$;
\item if $r=2$, $\nu_2(a^n-1)=\nu_2(a^2-1)+\nu_2(n)-1$ if $n$ is even and $\nu_2(a^n-1)=\nu_2(a-1)$ if $n$ is odd. In particular, if $a\equiv 1\pmod 4$, $\nu_2(a^n-1)=\nu_2(a-1)+\nu_2(n)$.
\end{enumerate}
\end{lemma}

From the previous lemma, we have the following result.

\begin{proposition}\label{prop:estimate-phi}
Let $a$ be a positive integer not divisible by the primes in the set $\mathcal C=\{r_1, \ldots, r_u\}$. Then there exist constants $L(\mathcal C, a), U(\mathcal C, a)>0$ (only depending on $a$ and the primes $r_i$) such that, for any positive integer $b$ whose set of prime factors is $\mathcal C$, the following holds:
$$L(\mathcal C, a)\le \frac{\varphi(b)}{\ord(a, b)}\le U(\mathcal C, a).$$
\end{proposition}

\begin{proof}
Set $R=\prod_{i=1}^ur_i$, $S=\ord(a, R)$ and $e_i=\nu_{r_i}(a^S-1)$, where $1\le i\le u$. Let $b=r_1^{E_1}\cdots r_u^{E_u}$ be a positive integer with $E_i\ge 1$ for any $1\le i\le u$. 
If $b$ is odd or $a\equiv 1\pmod 4$, Lemma~\ref{lem:LTE} entails that
$$\ord(a, b)=S\cdot \prod_{1\le i\le u}r_i^{m_i},$$
where $m_i=\max\{E_i-e_i, 0\}$. Therefore,

$$\frac{\varphi(b)}{\ord(a, b)}=\frac{\varphi(R)}{S}\cdot \prod_{1\le i\le u}r_i^{m_i^*-1},$$
where $m_i^*=E_i-m_i=\min\{E_i, e_i\}$. From the definition of $m_i^*$, we have that $1\le m_i^*\le e_i$ and so 
$$ 1\le \prod_{1\le i\le u}r_i^{m_i^*-1}\le \prod_{1\le i\le u}r_i^{e_i-1}\le \frac{a^S-1}{R}.$$
The result follows with $L(\mathcal C, a)=\frac{\varphi(R)}{S}$ and $U(\mathcal C, a)=\frac{\varphi(R)(a^S-1)}{RS}$. For $b$ even and $a\equiv 3\pmod 4$, we observe that $A=a^2$ satisfies $A\equiv 1\pmod 4$ and
$$\ord(A, b)\le \ord(a, b)\le 2\cdot \ord(A, b).$$
In this case, the result follows with $L(\mathcal C, a)=\frac{L(\mathcal C, A)}{2}$ and $U(\mathcal C, a)=U(\mathcal C, A)$.
\end{proof}

The following result provides sharp estimates on the growth on the functions $M_{f, g}(n), m_{f, g}(n)$ and  $N_{f, g}(n)$ in the case $g$ a monomial.

\begin{theorem}\label{thm:monomial-growth}
Let $f\in \F_q[x]$ be an irreducible polynomial of degree $k$, not of the form $ax$ such that $e=\ord(f)$. Let $D>1$ be a positive integer such that $\gcd(D, p)=1$ and write $D=d_1d_2$, where $\gcd(d_1, e)=1$ and each prime factor of $d_2$ divides $e$. Then, for $g(x)=x^D$, the following hold:

\begin{enumerate}[(i)]
\item $M_{f, g}(n)\approx D^n$;
\item $m_{f, g}(n)\approx d_2^n$;
\item $N_{f, g}(n)\approx n^t$, where $t$ is the number of distinct prime factors of $d_1$.
\end{enumerate}
\end{theorem}

\begin{proof}
From Proposition~\ref{prop:estimate-phi} and item (ii) of Corollary~\ref{cor:monomial}, we have that $M_{f, g}(n)\approx \varphi(D^n)$. Since
$$\frac{\varphi(D^n)}{D^n}=\frac{\varphi(D)}{D},$$
we have that $\varphi(D^n)\approx D^n$ and so $M_{f, g}(n)\approx D^n$. Similarly we obtain $m_{f, g}(n)\approx d_2^n$. It remains to prove item (iii). Let $C_1$ and $C_2$ be the set of distinct prime factors of $d_1$ and $d_2$, respectively, and let $\mathcal S$ be the set of all non-empty subsets of $C_1\cup C_2$. For each $\mathcal C\in \mathcal S$, let $L(\mathcal C, q)$ and $U(\mathcal C, q)$ be as in Proposition~\ref{prop:estimate-phi} and set
$L=\min\limits_{\mathcal C\in \mathcal S}L(\mathcal C, q)$ and $U=\max\limits_{\mathcal C\in \mathcal S}U(\mathcal C, q)$. From Proposition~\ref{prop:estimate-phi}, we have that
$$L\le \frac{\varphi(Md_2^n)}{\ord(q, Md_2^n)}\le U,$$
where $n\ge 1$ and $M$ is any divisor of $d_1^n$.
Since $\ord(q, Md_2^ne)\le \ord(q,e)\cdot \ord(q, Md_2^n)$, there exist $L', U'>0$ such that, for any $n\ge 1$ and any divisor $M$ of $d_1^n$, the following holds:
$$L'\le \frac{k\cdot \varphi(Md_2^n)}{\ord(q, Md_2^ne)}\le U'.$$
Since $M$ and $d_2$ are relatively prime, we have that $$\varphi(Md_2^n)=\varphi(M)\cdot \varphi(d_2^n)=\varphi(M)\cdot d_2^n\cdot \frac{\varphi(d_2)}{d_2}.$$
In conclusion, there exist constants $c_0,c_1>0$ such that
$$c_0\le \frac{kd_2^n\varphi(M)}{\ord(q, {Md_2^ne})}\le c_1,$$
for any $n\ge 0$ and any divisor $M$ of $d_1^n$. Let 
$d_1=r_1^{m_1}\ldots r_t^{m_t}$ be the prime factorization of $d_1$. In particular, $d_1^n$ has 
$$(nm_1+1)\cdots (nm_t+1)\approx n^t$$
distinct divisors. Therefore, from item (v) of Corollary~\ref{cor:monomial}, we have that
$$N_{f, g}(n)=\sum_{M|d_1^n}\frac{kd_2^n\varphi(M)}{\ord(q, {Md_2^ne})}\approx n^t.$$
\end{proof}

\begin{remark}
Let $f$, $g$ and $d_2$ be as in the previous theorem and let $\alpha\in\overline{\F}_q$ be any root of $f$. From the previous theorem, we have that $m_{f, g}(n)\to \infty$ unless $d_2=1$, i.e., $D$ is relatively prime with the order $e$ of $f(x)$. It is direct to verify that the latter holds if and only if $\alpha$ is not $g$-periodic, as predicted by Theorem~\ref{thm:main2}.
\end{remark}

The following corollary is an immediate application of Theorem~\ref{thm:monomial-growth} and provides the proof of item (i) in Theorem~\ref{thm:main3}.

\begin{corollary}
Let $f\in \F_q[x]$ be a polynomial of degree at least one that is not of the form $ax^k$. Write
$$f(x)=ax^lf_1(x)\cdots f_m(x),$$
where $l\ge 0$, $m\ge 1$ and each $f_i(x)\ne x$ is irreducible over $\F_q$ with order $e_i=\ord(f_i)$. 
Let $\mathcal C$ be the set of distinct prime divisors of the numbers $e_i$. If $D>1$ is any positive integer with $t$ distinct prime factors, none of them being an element of $\mathcal C$ or the prime $p$, then for $g(x)=x^D$ the following holds:
$$M_{f, g}(n)\approx D^n\quad\text{and}\quad N_{f, g}(n)\approx n^t.$$
\end{corollary}

\subsection{The linearized case}
Here we consider $g(x)=\sum_{i=0}^ma_ix^{q^i}$ a $q$-linearized polynomial, where $g'(x)$ is not the zero polynomial and $f\in \F_q[x]$ is an irreducible polynomial. As pointed out in~\cite{R18}, we have an analogue of Theorem~\ref{thm:butler} for $q$-linearized polynomials with suitable changes. First, in contrast to the order of polynomials, we introduce the $\F_q$-order. 

\begin{definition}
Let $\alpha\in \overline{\F}_q$ and $g\in \F_q[x]$.

\begin{enumerate}[(a)]
\item if $g(x)=\sum_{i=0}^ma_ix^i$, the polynomial  $L_g(x)=\sum_{i=0}^ma_ix^{q^i}$ is the linearized $q$-associate of $g$;
\item the $\F_q$-order of $\alpha$ is the monic polynomial $h\in \F_q[x]$ of least degree such that its $q$-associate $L_h$ vanishes at $\alpha$, i.e., $L_h(\alpha)=0$. We write $h(x)=m_{\alpha, q}(x)$.
\end{enumerate}
\end{definition}

The following lemma compiles some basic properties of the $q$-associates and the $\F_q$-order. Its proof follows from the results in Subsection~2.1 of~\cite{R18} and so we omit the details.

\begin{lemma}\label{lem:properties-q}
The following hold:
\begin{enumerate}[(i)]
\item the $\F_q$-order of an element $\alpha$ is well defined, is not divisible by $x$ and coincides with the $\F_q$-order of any conjugate $\alpha^{q^i}$ of $\alpha$;
\item for polynomials $g_1, g_2\in \F_q[x]$, we have that
$$L_{g_1}(L_{g_2}(x))=L_{g_1g_2}(x).$$
\end{enumerate}
\end{lemma}

We introduce the analogue of the Euler's Totient function for the polynomial ring $\F_q[x]$.
\begin{definition}
The \emph{Euler's Totient function} for polynomials over $\F_q$ is $$\Phi_q(f)=\left |\left(\frac{\F_q[x]}{\langle f\rangle}\right)^{*}\right |,$$ where $\langle f\rangle$ is the ideal generated by $f$ in $\F_q[x]$.
\end{definition}
According to Theorem~4 of~\cite{R18}, we have the following result.
\begin{theorem}\label{thm:main-reis}
Let $f\in \F_q[x]$ be an irreducible polynomial of degree $k$ such that any of its roots has $\F_q$-order $h$. Let $g\in \F_q[x]$ be a monic polynomial such that $\gcd(g(x), x)=1$ and write $g=g_1g_2$, where $\gcd(g_1, h)=1$ and each irreducible factor of $g_2$ divides $h$. If $L_g$ denotes the $q$-associate of $g$ and $\deg(g_2)=m$, then
for each monic divisor $G$ of $g_1$, $f(L_g(x))$ has exactly $$\frac{kq^m\Phi_q(G)}{\ou(x, Gg_2h)},$$ irreducible factors of degree $\ou(x, Gg_2h)$ with roots of $\F_q$-order $Gg_2h$. In addition, this describes all the irreducible factors of $f(L_g(x))$ over $\F_q$.
\end{theorem}

\begin{remark}
We observe that, for any $a\in \F_q^*$ and $g\in \F_q[x]$, the following holds $$aL_g(x)=L_g(ax)=L_{ag}(x).$$ Therefore, the $\F_q$-order of $\alpha\in \overline{\F}_q$ and $a\alpha$ coincide whenever $a\in \F_q^*$. In particular, Theorem~\ref{thm:main-reis} holds without the assumption that $g$ is monic. This fact is frequently used.
\end{remark}

From item (ii) of Lemma~\ref{lem:properties-q}, we directly obtain a formula for iterates of $q$-linearized polynomials:

\begin{equation}\label{eq:iterates-q}
L_g^{(n)}(x)=L_{g^n}(x), \; n\ge 0.
\end{equation}

We observe that the condition $\gcd(g(x), x)=1$ in Theorem~\ref{thm:main-reis} is compatible with our initial assumption; in fact, $\gcd(g(x), x)=1$ if and only if the derivative of $L_g(x)$ is not the zero polynomial. From Theorem~\ref{thm:main-reis} and Eq.\eqref{eq:iterates-q}, we easily obtain the following analogue of Corollary~\ref{cor:monomial}.

\begin{corollary}\label{cor:linearized}
Let $f\in \F_q[x]$ be an irreducible polynomial of degree $k$ such that any of its roots has $\F_q$-order $h$. Let $g\in \F_q[x]$ be a polynomial of degree $D\ge 1$ such that $\gcd(g(x), x)=1$ and write $g=g_1g_2$, where $\gcd(g_1, h)=1$ and each irreducible factor of $g_2$ divides $h$. If $L_g$ denotes the $q$-associate of $g$ and $\deg(g_2)=m$, then for any $n\ge 0$, the following hold:
\begin{enumerate}[(i)]
\item $E_{f, L_g}(n)=e_{f, L_g}(n)=1$;
\item $\Delta_{f, L_g}(n)=k\cdot q^{Dn}$;
\item $M_{f, L_g}(n)=\ou(x, g^nh)$;
\item $m_{f, L_g}(n)=\ou(x, g_2^nh)$;
\item $N_{f, L_g}(n)=\sum_{G|g_1^n}\frac{kq^{mn}\Phi_q(G)}{\ou(x, {Gg_2^nh})},$ where $G$ is monic and polynomial division is over $\F_q$.
\end{enumerate}
\end{corollary}
In analogy to the monomial case, the computation of orders $\ou(x, F)$ is required. This is done with the help of the following result.

\begin{lemma}\label{lem:order-mult}
Let $F\in \F_q[x]$ be a non constant polynomial that is not divisible by $x$ and let $\mathrm{rad}(F)$ be the squarefree part of $F$. Then the following hold:
$$\ou(x, F)=\ou(x, \mathrm{rad}(F))\cdot p^r,$$
where $r=\lceil \log_p(\nu(F))\rceil$.
\end{lemma}

\begin{proof}
For the proof of this result, see item (ii) of Lemma~5 in~\cite{R18}.
\end{proof}

\begin{proposition}\label{thm:linearized-growth}
Let $f\in \F_q[x]$ be an irreducible polynomial of degree $k$ such that any of its roots has $\F_q$-order $h$. Let $g\in \F_q[x]$ be a polynomial of degree $D\ge 1$ such that $\gcd(g(x), x)=1$ and write $g=g_1g_2$, where $\gcd(g_1, h)=1$ and each irreducible factor of $g_2$ divides $h$. If $L_g$ denotes the $q$-associate of $g$, then $M_{f, L_g}(n)\approx n$. In addition, if $g_2(x)\ne 1$, then 
$$A_{f, L_g}(n)\approx n\quad\text{and}\quad N_{f, L_g}(n)\approx \frac{q^{Dn}}{n}.$$
\end{proposition}

\begin{proof}
Combining item (iii) of Corollary~\ref{cor:linearized} and Lemma~\ref{lem:order-mult}, we have that
$$M_{f, L_g}(n)=\ou(x, g^nh)=\ou (x, \mathrm{rad}(gh))\cdot p^{t(n)},\; n\ge 1,$$
where $t(n)=\lceil \log_p(\nu(g^nh))\rceil$. If we set $t_0=\nu(g)$, for $n$ sufficiently large, we have that $\nu(g^nh)=nt_0+R$ for some $R\ge 0$ not depending on $n$. Therefore, 
$$p^{t(n)}\approx p^{\log_p(t_0n)}\approx n.$$
In particular, since $N_{f, L_g}(n)\cdot M_{f, L_g}(n)\ge \Delta_{f, L_g}(n)=kq^{Dn}$, we have that $$N_{f, L_g}(n)\gg \frac{q^{Dn}}{n}.$$ If $g_2(x)\ne 1$, we follow the previous arguments and obtain that
 $$m_{f, L_g}(n)=\ou(x, g_2^nh)\approx n.$$
 Since $N_{f, L_g}(n)\cdot m_{f, L_g}(n)\le \Delta_{f, L_g}(n)=kq^{Dn}$, we have that $$N_{f, L_g}(n)\ll \frac{q^{Dn}}{n}.$$
 Therefore, $N_{f, L_g}(n)\approx \frac{q^{Dn}}{n}$. Since $A_{f, L_g}(n)\cdot N_{f, L_g}(n)=\Delta_{f, L_g}(n)\approx q^{Dn}$, we have that
 $A_{f, L_g}(n)\approx n$.
\end{proof}

The following corollary is an immediate application of Proposition~\ref{thm:linearized-growth} and provides the proof of item (ii) in Theorem~\ref{thm:main3}.

\begin{corollary}
Let $f\in \F_q[x]$ be a polynomial of degree at least one and let $g\in \F_q[x]$ be a polynomial of degree $D\ge 1$ such that $\gcd(g(x), x)=1$. If $L_g$ denotes the $q$-associate of $g$, then
$M_{f, L_g}(n)\approx n$.
\end{corollary}

\subsection{Pairs $(f, g)$ for which $N_{f, g}(n)$ and $M_{f, g}(n)$ have exponential growth}
Here we provide the proof of item (iii) in Theorem~\ref{thm:main3}. From Lemma~\ref{lem:reduction-1}, it suffices to consider $f\in \F_q[x]$ an irreducible polynomial such that $f(x)=x^k+\sum_{i=0}^{k-1}a_ix^i$ with $a_0\ne 0, 1$. We have the following result.

\begin{proposition}
Let $q>2$ be a prime power and let $f\in \F_q[x]$ be an irreducible polynomial such that $f(x)=x^k+\sum_{i=0}^{k-1}a_ix^i$ with $a_0\ne 0, 1$. For $g(x)=(x^q-x)^{q-1}$, the polynomial $f(g(x))$ is separable, reducible and any of its irreducible factors has degree of the form $dk$ with $d\ge 2$. In particular, $f(g(x))$ has an irreducible factor of the form $h(x)=x^{dk}+\sum_{i=0}^{dk-1}b_ix^i$ with $b_0\ne 0, 1$ and $d\ge 2$. 
\end{proposition}
\begin{proof}
We first prove that $f(g(x))$ is separable. For this we observe that, since $a_0=f(0)\ne 0$, the polynomial $f(g(x))$ does not have roots in $\F_q$. In particular, since the formal derivative of $f(g(x))$ equals $f'(g(x))\cdot (x^q-x)^{q-2}$, any repeated root of $f(g(x))$ is also a root of $f'(g(x))$. Since $f$ is irreducible and $\F_q$ is a perfect field, we have that $\gcd(f(x), f'(x))=1$ and so $\gcd(f(g(x)), f'(g(x)))=1$. Therefore, $f(g(x))$ cannot have repeated roots.

Let $\beta\in\overline{\F}_q$ be a root of $f(g(x))$ and $B$ its degree over $\F_q$. Corollary~\ref{cor:spin} entails that $B$ is divisible by $k$, the degree of $f(x)$. If we write $B=dk$ with $d\ge 1$, we just need to prove that $d\ne 1, q(q-1)$.
\begin{enumerate}
\item If $d=q(q-1)$, $f((x^q-x)^{q-1})$ is an irreducible polynomial. However, using Theorem 5 and Corollary~1 of~\cite{R18}, we easily see that any polynomial of the form $f((x^q-x)^{q-1})$ is reducible over $\F_q$ if $q>2$.

\item If $d=1$, we have that $\beta\in \F_{q^k}$ and so $\beta^q-\beta\in \F_{q^k}$. Since $\beta$ is a root of $f(g(x))$, there exists a root $\alpha\in \F_{q^k}$ of $f(x)$ such that $(\beta^q-\beta)^{q-1}=\alpha$. However, $\alpha^{\frac{q^k-1}{q-1}}=a_0\ne 1, 0$ and so $\alpha$ cannot be of the form $\gamma^{q-1}$ for any $\gamma\in \F_{q^k}$.
\end{enumerate}

The statement regarding the existence of an irreducible factor $h(x)$ as above follows from the fact that $a_0\ne 0, 1$. 
\end{proof}

The previous proposition immediately gives the following result.

\begin{corollary}
Let $q>2$ be a prime power and let $f\in \F_q[x]$ be an irreducible polynomial of degree $k$ such that $f(x)=x^k+\sum_{i=0}^{k-1}a_ix^i$ with $a_0\ne 0, 1$. For $g(x)=(x^q-x)^{q-1}$ and any $n\ge 0$, we have that
$$N_{f, g}(n)\ge 2^n\quad\text{and}\quad M_{f, g}(n)\ge k\cdot 2^n.$$
\end{corollary}

\section{Conclusions and open problems}
This paper provided a study on the growth of some arithmetic functions related to the factorization of iterated polynomials $f(g^{(n)}(x))$ over finite fields, such as the number and the degree of the irreducible factors. This study extended and enhanced many results of~\cite{GOS}, where the case $f=g$ is considered; for more details, see Theorem~\ref{thm:shpa} and Corollary~\ref{cor:MAIN}. Here we provide some open problems and conjectures based on theoretical and computational considerations. 

Throughout this section we consider arbitrary pairs $(f, g)$ of polynomials over $\F_q$ that are neither critical nor $p$-critical, where $f$ is irreducible and $g$ has degree at least two. Theorem~\ref{thm:main3} entails that $N_{f, g}(n)$ may have polynomial growth of any degree, while $M_{f, g}(n)$ may reach the linear growth (showing that the bound $M_{f, g}(n)\gg n$ in Theorem~\ref{thm:main2} is optimal for a generic polynomial $f\in\F_q[x]$).

\begin{problem}\label{p1}
For an irreducible polynomial $f\in \F_q[x]$, find all the polynomials $g\in \F_q[x]$ such that $M_{f, g}(n)\approx n$.
\end{problem}

\begin{conjecture}\label{c1}
One of the following holds:
\begin{enumerate}[(i)]
\item $M_{f, g}(n)\approx n$;
\item $\log M_{f, g}(n)\gg n$.
\end{enumerate}
\end{conjecture}
Conjecture~\ref{c1} entails that $M_{f, g}(n)$ does not have polynomial growth of high degree (quadratic, cubic, etc). We believe that any proof (or disproof) of Conjecture~\ref{c1} contains a solution for Problem~\ref{p1}.

Remark~\ref{remark:main} entails that there exists $c>0$ such that, for any $n\ge 0$, either $N_{f, g}(n)\ge c\cdot d^{n/2}$ or $M_{f,g}(n)\ge c\cdot d^{n/2}$. However, this is not sufficient to conclude that $N_{f, g}(n)$ or $M_{f, g}(n)$ have exponential growth. Nevertheless, we believe that this is always the case.

\begin{conjecture}
Either $\log (N_{f, g}(n))\gg n$ or $\log (M_{f, g}(n))\gg n$.
\end{conjecture}

Finally, we propose a problem regarding the growth of $A_{f, g}(n)$.

\begin{problem}\label{p2}
Prove or disprove: $A_{f, g}(n)\gg n$.
\end{problem}

We comment that if the bound $A_{f, g}(n)\gg n$ holds, then it is optimal, since we have provided examples where $A_{f, g}(n)\approx n$ (see Proposition~\ref{thm:linearized-growth}).


\projects{\noindent The author was supported by FAPESP 2018/03038-2, Brazil.}

\end{document}